\documentclass[10pt]{article}

\usepackage{caption}

\usepackage{amsthm}

\newtheorem{thm}{\bf Theorem}[section]
\newtheorem{lem}[thm]{\bf Lemma}

\newtheorem{prop}[thm]{\bf Proposition}

\newtheorem{rem}[thm]{\bf Remark}

\DeclareCaptionType{NonFloatingFigureCaption}[Figure][Figure]

\def\ds{\displaystyle}

\def\eps{\varepsilon}
\def\ph{\varphi}

\usepackage{yhmath} 
\usepackage{amsfonts} 
\usepackage{mathtools} 

\DeclarePairedDelimiter{\ceil}{\lceil}{\rceil}
\DeclarePairedDelimiter\floor{\lfloor}{\rfloor}
\DeclareMathAlphabet{\mathpzc}{OT1}{pzc}{m}{it} 

\def\Area{\operatorname{Area}}

\def\Boundary{\partial}

\newcommand{\CeilingFunction}[1]{\ceil*{#1}}
\newcommand{\ClosedInterval}[1]{\left[#1\right]}

\def\ComplexPlane{\SetOfComplexNumbers}
\def\ComposedWith{\circ}
\newcommand{\ContinuedFractionFlatNotation}[1]{\left[{#1}\right]}

\def\Distance{d}

\def\EmptySet{\emptyset}

\newcommand{\FractionalPart}[1]{\left\{#1\right\}}

\def\ImaginaryPart{\mathpzc{Im}\,}
\newcommand{\IntegerPart}[1]{\floor*{#1}}
\newcommand\Interior[1]{\ring{#1}}

\newcommand{\MappingDomain}[1]{\mathcal{D}\left(#1\right)}

\newcommand{\OpenInterval}[1]{\left]#1\right[}
\newcommand{\OpenClosedInterval}[1]{\left]#1\right]}

\def\RealPart{\mathpzc{Re}\,}

\def\RiemannSphere{\widehat{\ComplexPlane}}

\def\SetComplement{\backslash}

\def\SetOfComplexNumbers{\mathbb{C}}
\def\SetOfIntegers{\mathbb{Z}}
\def\SetOfNaturalNumbers{\mathbb{N}}
\def\SetOfNonNegativeRealNumbers{\SetOfRealNumbers_+}

\def\SetOfPositiveNaturalNumbers{\mathbb{N}^*}

\def\SetOfRealNumbers{\mathbb{R}}

\newcommand{\SmallO}[1]{o\left(#1\right)}
\def\SquareRootOfMinusOne{i}

\def\TendsTo{\rightarrow}

\newcommand{\TopologicalBoundary}[1]{\partial #1}

\def\Disk{\mathbb{D}}

\def\AttractingTag{{a}}
\def\AttractingFatouCoordinates{\Phi^{\AttractingTag}}

\def\BottcherMap{B}
\def\BottcherMapInverse{\psi}
\def\BottcherInverseLaurentCoefficient{b}

\def\BPFPAttractingArc{\ell^{\AttractingTag}}
\def\BPFPAttractingCrescent{S^{\AttractingTag}}
\def\BPFPAttractingPetal{\Omega^{\AttractingTag}}
\def\BPFPAttractingPetalReferencePoint{z^{\AttractingTag}}
\def\BPFPDomainOfExtendedAttractingFatouCoordinates{B^{\AttractingTag}}
\def\BPFPDomainOfExtendedRepellingFatouCoordinates{B^{\RepellingTag}}
\def\BPFPRepellingArc{\ell^{\RepellingTag}}
\def\BPFPRepellingCrescent{S^{\RepellingTag}}
\def\BPFPRepellingPetal{\Omega^{\RepellingTag}}
\def\BPFPRepellingPetalReferencePoint{z^{\RepellingTag}}

\def\BPFPUnionOfAllPetals{\Omega}

\def\DoubleMandelbrotSet{\mathcal{DM}}

\def\FilledJuliaSet{K}

\def\GreenFunction{G}
\def\GreenFunctionSublevelSet{V}

\def\InverseRepellingFatouCoordinates{\varphi}

\newcommand{\Iterated}[1]{{\circ #1}}

\def\JuliaSet{J}

\def\MandelbrotSet{\mathcal{M}}

\def\PreFatouCoverMap{\tau}

\def\QuadraticPolynomialC{P}
\newcommand{\QuadraticPolynomialCFormula}[2]{#2^2 + #1}
\def\QuadraticPolynomialLambda{Q}
\newcommand{\QuadraticPolynomialLambdaFormula}[2]{#1 #2 + #2^2}

\def\RepellingTag{r}
\def\RepellingFatouCoordinates{\Phi^{\RepellingTag}}

\def\area{\Area}
\def\bott{\BottcherMap}
\def\bottinv{\BottcherMapInverse}
\def\bottinvcoeff{\BottcherInverseLaurentCoefficient}
\def\ceilf{\CeilingFunction}
\def\cplxp{\ComplexPlane}
\def\critp{\CriticalPoint}
\def\disk{\Disk}
\def\dblmand{\DoubleMandelbrotSet}

\def\fillj{\FilledJuliaSet}

\def\green{\GreenFunction}
\def\greensl{\GreenFunctionSublevelSet}
\def\ii{\SquareRootOfMinusOne}
\def\interior{\Interior}
\def\iter{\Iterated}
\def\julia{\JuliaSet}

\def\qpc{\QuadraticPolynomialC}
\def\qpcf{\QuadraticPolynomialCFormula}
\def\qpl{\QuadraticPolynomialLambda}
\def\qplf{\QuadraticPolynomialLambdaFormula}
\def\riesph{\RiemannSphere}
\def\reals{\SetOfRealNumbers}
\def\realsnn{\SetOfNonNegativeRealNumbers}
\def\tendsto{\TendsTo}
\def\wo{\SetComplement}

\def\FakeCremerExponent{16}

\def\BottcherDomain{U}
\def\BottcherRangeComplementaryDisk{D}

\def\bic{\bottinvcoeff}

\def\AlphaNeighbourhoodBound{{\alpha_0}}
\def\AreaGap{C}
\def\AreaFormulaPartialSum{A}
\def\AreaFormulaPartialSumUpperBound{N}
\def\BigIterate{p}

\def\CriticalPoint{c}
\def\DomainOfInverseRepellingFatou{H}
\def\DomainOfInverseRepellingFatouMaxRealPart{\xi}
\def\EggBeaterPhase{\tau}
\def\EscapingDomain{U}
\def\EscapeRadius{R}
\def\IterationsToThePetal{m}
\def\LavaursMap{L}
\def\LogGrowthOrder{\gamma}
\def\MaxDifferenceInRealPartFatouCoordinate{\chi}

\def\TheOtherFixedPoint{\sigma}
\def\TransitTimeMinusInverseRotationNumberBound{P}

\newcommand{\greencrit}[1]{\mathcal{N}_{#1}}

\def\areapartsum{\AreaFormulaPartialSum}
\def\bigiterate{\BigIterate}
\def\DomainOf{\MappingDomain}
\def\domainof{\DomainOf}
\def\ebphase{\EggBeaterPhase}
\def\escrad{\EscapeRadius}
\def\sumend{\AreaFormulaPartialSumUpperBound}

\def\lis{\sumend}
\def\lit{\bigiterate}
\def\pa{\areapartsum}

\usepackage{hyperref}

\usepackage{tikz}
\usepackage{gnuplot-lua-tikz}

\begin{document}

\title{Convergence properties of the Gronwall area formula for quadratic Julia sets}
\author{Alexandre Dezotti}
\date{}


\maketitle
\begin{abstract}
Gronwall area formula can be used to express
 the area of the filled Julia set of connected quadratic Julia sets.
Using parabolic enrichment, it is shown that there cannot be a nice approximation of this formula by partial sums
 which is  uniform along the boundary of the main cardioid of the Mandelbrot set.
\end{abstract}

\section{Introduction}

One important aspect of the study of the iteration of holomorphic functions is the study of the Julia set.
The Julia set of a holomorphic mapping corresponds to the chaotic part of the dynamical system consisting
in iterating the corresponding mapping.

The question of the measure of Julia sets is a very natural question, already raised by Fatou in his memoir
 \cite{Fatou1919}, p.243.
For some time it has been possible to conjecture that all of the Julia sets of quadratic polynomials have zero area.
Indeed this is the case in many situations
 (see, for example \cite{DouadyHubbard1985}, \cite{Lyubich1991}, \cite{PetersenZakeri2004})
 and, if this conjecture were true,
 it would have implied the conjecture
 stating that hyperbolic dynamical systems are dense in the family of complex quadratic polynomials
\footnote{
This conjecture could, in some way, be dated back to Fatou's discussion \cite{Fatou1920a}, p.73.
Nevertheless, a likely interpretation of his discussion would lead to a slightly different conjecture,
 which has been proven false.
On this matter, compare McMuller \cite{McMullenBook1994}, chapter 4.}.
No proof of the latter conjecture has been published yet, except for real polynomials
(\cite{GraczykSwiatek1997} or \cite{GraczykSwiatekBook1998}),
 but in their paper \cite{BuffCheritat2012},
 Buff and Ch\'eritat showed the existence of quadratic polynomials for which the Julia set has positive area.
Hence the proof of the hyperbolicity density conjecture cannot relies on this argument.

Interestingly, some of those examples are period one Cremer parameters.
 Recall that, in the quadratic family, a period one Cremer parameter is a value of $c\in\SetOfComplexNumbers$
 such that the polynomial $\qpc_c(z)=z^2+c$
 has a non hyperbolic fixed point \footnote{Which means, here, by a slight abuse of terminology,
 that the derivative has modulus one.}
 on the neighbourhood of which $\qpc_c$
 cannot be conjugated to its linear part and such that the linear part is not periodic.
Those parameters lie in the boundary of the main cardioid of the Mandelbrot set.
An equivalent definition of Cremer parameters is the following.
 If $z_0$ denotes the corresponding fixed point,
 $|\qpc_c'(z_0)|=1$, $\qpc_c'(z_0)$ is not a root of unity and $\qpc_c$ is not linearisable around $z_0$.
The point $z_0$ is called a (period one)
 Cremer point and one can also talk about periodic Cremer point in a simliar way for any period
 \footnote{And also the fact that the function is a quadratic polynomial has nothing to do with the general definition,
 valid for any holomoprhic function around $z_0$.}.

For polynomials,
 one can define the filled Julia set
 which is the union of the Julia set and all the bounded component of its complement in the complex plane.
The filled Julia sets of other parameters such as Siegel parameters and parabolic parameters have positive area.
The Siegel parameters correspond to the linearisable case of a non hyperbolic fixed (or periodic) point
 and parabolic parameters are the values of $c$ such that $\qpc_c$ has a periodic point $z_0$ for which
 $\left(\qpc_c^{\Iterated p}\right)'(z_0)$ is a root of unity where $p$ denotes the period of $z_0$.
In both case, the filled Julia set has non empty interior.

As the boundary of the main cardioid correponds to the set of quadratic polynomials having a non hyperbolic fixed point,
 Hubbard asked whether there would be a lower bound on the area of the filled Julia set along the boundary of the main cardioid.
 For example, the question is stated for the Cremer parameters in \cite{Holbaek2012}.
The question is more critical for Cremer parameters as we do not know if some of them have a Julia set with zero area.
Indeed, the complexity of the situation is reflected by the fact that,
 due to the ``close proximity'' of Cremer parameters to parabolic parameters,
 standard algorithms cannot distinguish between Cremer parameters and parabolic parameters,
 see for example figure \ref{fig:cremer julia set?}.

It is well known that Cremer parameters are generic on the boundary of a hyperbolic component of the Mandelbrot set
\cite{B-Milnor2006}.
 Moreover Lyubich proved that for generic parameters (in the sense of Baire)
 on the boundary of the Mandelbrot set the area of the Julia set is zero \cite{Lyubich1983}.
Unfortunately this is insufficient to prove that some of this Cremer Julia set have zero area,
 as the complement of the union of the boundaries of hyperbolic components
 is a generic subset of the boundary of the Mandelbrot set.

\begin{figure}
\begin{center}
\includegraphics[width=0.42\textwidth]{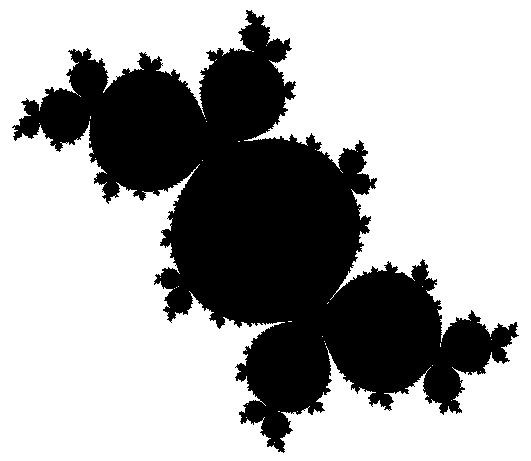}
\includegraphics[width=0.42\textwidth]{./ensemble_de_julia_rempli_du_lapin_gras.png}
\end{center}
\caption{\label{fig:cremer julia set?}
Computer representations of the filled Julia sets of $\lambda z +  z^2$ for 
 the parameter $\lambda=e^{2\pi\SquareRootOfMinusOne/3}$ of the fat Douady's rabbit on the left
 and, on the right, for the Cremer parameter
$\lambda=e^{2\pi\SquareRootOfMinusOne\theta}$, with
$\theta=\ContinuedFractionFlatNotation{0, 3, 10^{\FakeCremerExponent}, a_3, \dots, a_n, \dots }$
where $\ContinuedFractionFlatNotation{a_0,a_1,a_2,\dots}$
 is the notation for the continued fraction
 $a_0+1/(a_1+1/(a_2+\dots))$
 with coefficients $a_n$,
 and
$a_n$ satisfies, for $n\geq 2$, $a_{n+1}=2^{q_n}$, where $q_n$ is the denominator of the $n$th convergent.
}
\end{figure}

This article concerns the Gronwall area formula applied to the area of\linebreak
 quadratic filled Julia sets. Let's recall its content.

For a closed subset of the Riemann sphere $\RiemannSphere$
 containing at least two points and  whose complement is simply connected and contains $\infty$,
 the Gronwall area formula provides a way to compute its area, that is its Lebesgue two dimensional measure.
Namely\cite{Gronwall1914}, if
 $\BottcherMapInverse(w)=\ds{\sum_{n\leq 1}}\bic_{n}w^n$
 is the Laurent series of a conformal isomorphism between the complement of the disk of radius
 $r_0\geq 1$ centred at $0$ in $\RiemannSphere$
 and the complement of the compact set $\FilledJuliaSet\subset\SetOfComplexNumbers$, then, for any $r\geq r_0$,
\begin{equation}\label{eq:original area formula}
\Area \left(\left\{z:|\BottcherMapInverse^{-1}(z)|\leq r\right\}\right)
 = 
\pi\sum_{n\leq 1}n|\bic_{n}|^2r^{2n}.
\end{equation}
Gronwall area formula is essentially a consequence of Green's theorem.

As a consequence,
 this result yields an expression for the area of the filled Julia set of any polynomial map with connected Julia set.
Indeed, by a classical theorem of B\"ottcher (compare for example \cite{B-Milnor2006}),
 the dynamics of any polynomial of degree $d$ is conjugated to the dynamics of
 $w\mapsto w^d$ near $\infty$.
 This provides a natural isomorphism between the basin of attraction of $\infty$ and the complement of the unit disk.

It is possible to explicitly compute
the coefficients of the Laurent series of the inverse of the B\"ottcher map of a polynomial
 and then use them in order to numerically evaluate an approximation of the formula \eqref{eq:original area formula}
 via a finite summation.

We are particularly interested in the case of quadratic polynomials.
If we denote by $\GreenFunction_\lambda$ the Green function,
 i.e. the logarithm of the modulus of the B\"ottcher map
 of the quadratic polynomial $\qpl_\lambda(z)=\qplf{\lambda}{z}$,
we can define an approximation of the area of the sublevel set
 $\{z\in\SetOfComplexNumbers:\GreenFunction_\lambda(z)\leq \log r\}$
 by a finite sum with
\begin{equation}\label{eq:intro def area partial sum}
 \AreaFormulaPartialSum (\lambda,r,N)=\pi\ds{\sum_{n=-N}^{1}}n|\bic_{n}|^2r^{2n}.
\end{equation}

It is
 natural to ask how close the value of $\AreaFormulaPartialSum (\lambda,r,\lis)$
 is to the actual area of the sublevel set of the Green function,
 and, more interestingly, how close is the value of $\AreaFormulaPartialSum (\lambda,1,\lis)$
 to the area of the filled Julia set?

An example of a numerical computation using the Gronwall area formula is given in the figure
\ref{fig:area numerical computation}
 (see also \cite{B-Milnor2006}, appendix A).
This example shows, for quadratic polynomials in the form $\qpc_c(z)=\qpcf{c}{z}$,
 the value of the approximation when the parameter $c$
 varies on the upper half of the boundary of the main cardioid of the Mandelbrot set.
Similar methods can be applied to compute an estimate of the area of the Mandelbrot set,
 see, for example, the work of Ewing and Schober \cite{EwingSchober1992}
 where they also compare the result with pixel counting methods
 and lower bounds using the area of the biggest hyperbolic components.

\begin{figure}
\begin{center}
\scalebox{0.7}{
\input{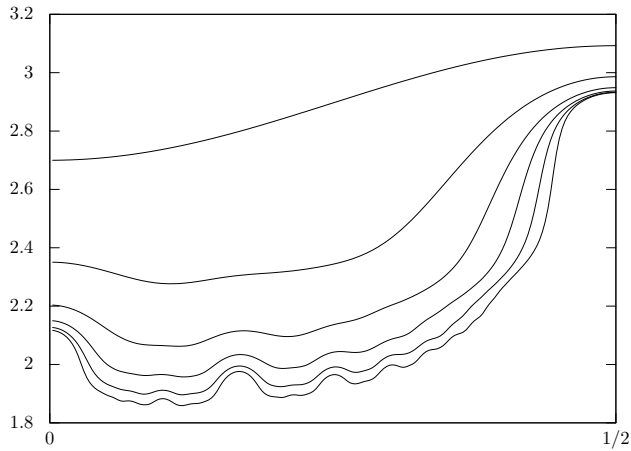}
}
\end{center}
\caption{
\label{fig:area numerical computation}
Plot of computed values of the truncated area formula for filled Julia sets ($r=1$)
along the upper boundary of the main cardioid.
The different graphs represents different level of truncations:
 $1$, $20$, $200$, $2000$, $20000$ and $200000$ terms
 (due to the definition of $\BottcherMap_c$, half of the terms are $0$).
The result of the computations decreases as the level increases.
The values on the horizontal axis represent the rotation number.
}
\end{figure}

Testing the lower bound hypothesis with numerical experiments requires to know the answer to the previous questions.
While it can give the impression of a lower bound,
 the figure \ref{fig:area numerical computation}
 gives some hint of slow convergence of the terms appearing in the sum for some parameters,
 for example on the side of the $1/2$ hill on the right.
This could be an indication that one cannot relies on such numerical experiments
for investigating the existence of a lower bound on the area.

Indeed, the theorem below states that the area of the filled Julia set is discontinuous near the parabolic parameter $1$.
This prevents a uniform approximation of the area by the formula.

In the following statement, mappings on the form $\qpl_\lambda(z)=\qplf{\lambda}{z}$
 are considered with $\lambda\in\SetOfComplexNumbers$.
 Those are conjugated to the family of mappings $\qpc_c(z)=\qpcf{c}{z}$ by affine maps,
 the correpondance between the two parameters being $c=\frac{\lambda}{2}\left(1-\frac{\lambda}{2}\right)$.
The connectedness locus of the family $(\qpc_c)_c$ is the Mandelbrot set,
 its counterpart for the family $(\qpl_\lambda)_\lambda$, called the double Mandelbrot set 
 (\cite{B-Milnor2006}, figure 29),
 is represented on figure \ref{fig:the double mandelbrot set}.

\begin{thm}\label{thm:discontinuity of area}
Let $\FilledJuliaSet_\lambda$ denote the filled Julia set of the quadratic polynomial
 $\qpl_\lambda(z)=\qplf{\lambda}{z}$.

Then
\begin{equation*}
\limsup_{\lambda\TendsTo 1,|\lambda|=1}\Area(\FilledJuliaSet_\lambda)<\Area(\FilledJuliaSet_1),
\end{equation*}
\end{thm}

In a more precise way, the theorem below shows that approximation of the area using a truncated Gronwall area formula
 fails to provide insight into the complexity of the variation of the filled Julia set.
Even if the number of terms used in finite sum approximations increases at a very fast rate as the parameter approaches $1$,
 there will be a definite discrepancy between the numerical result and the actual value of the area.

In what follows, we use the notation $\AreaFormulaPartialSum(\lambda,r,\lis)$
for the approximation of the area of Green sublevels given by
 \eqref{eq:intro def area partial sum}
 and $\FractionalPart{x}$ denotes the fractional part of the real number $x$,
 that is, $\FractionalPart{x}=x-n$ where $n$ is the largest integer smaller than or equal to $x$.

\begin{thm}\label{thm:too slow converges to K1}
 If $(\lis_n)_n$ is a sequence of natural numbers converging to $+\infty$ such that
 $\lis_n\in\SmallO{2^{1/\alpha_n}}$,
 then
\begin{equation*}
 \pa(\lambda_n, 1, \lis_n) \TendsTo \Area(\FilledJuliaSet_1).
\end{equation*}

\end{thm}

\begin{thm}\label{thm:the actual main theorem}

For any $\LogGrowthOrder>0$ 
and $\EggBeaterPhase\in\ClosedInterval{0,1}$, 
 there exists $C>0$ satisfying the following.

Let $(\alpha_n)_n$ be a sequence of positive real numbers
 converging to $0$
 and such that $\FractionalPart{\frac{1}{\alpha_n}}\TendsTo \EggBeaterPhase$,
 and let $(\lis_n)_n$ be a sequence of natural numbers.
Define $\lambda_n=e^{2\ii\pi\alpha_n}$ and
 suppose that $\log \lis_n\leq \frac{\LogGrowthOrder}{\alpha_n}$.

Then, there exists $n_0\in\SetOfNaturalNumbers$ such that, for all $n\geq n_0$,
\begin{equation*}
 \pa(\lambda_n, 1, \lis_n) \geq C + \area{K_{\lambda_n}}.
\end{equation*}

\end{thm}

Suppose we are to  compute an approximation of the area of the filled Julia set of a parameter close to $1$.
The above means that,
 in order to see any difference with the area of $\FilledJuliaSet_1$,
 it requires a number of terms which increases at least as fast as the exponential of the inverse of the distance
 of the parameter to $1$.
 
Thanks to the continuity of the straightening map of Mandelbrot-like family
 and to the uniform estimate on area distortion by quasiconformal maps in \cite{Astala1994},
 the theorems \ref{thm:too slow converges to K1} and \ref{thm:the actual main theorem}
 are also true for the cusp and the boundary of the main component of any Mandelbrot-like families
 and in particular of primitve copy of the Mandelbrot set itself.
 
Finally, techniques similar to those used in the proofs allow, with some minor modifications,
 the extension of those results to the root of any hyperbolic components of the Mandelbrot set.
 See, e.g. \cite{Shishikura1998}, section 7.

\begin{figure}
\begin{center}
\includegraphics[width=0.92\textwidth]{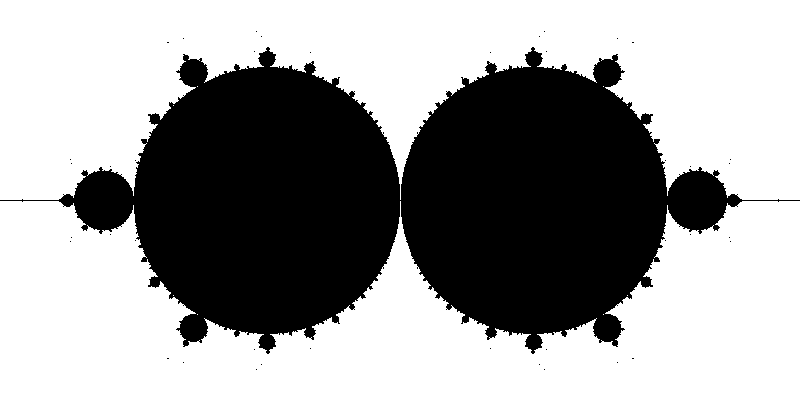}
\end{center}
\caption{\label{fig:the double mandelbrot set}
The double Mandelbrot set.
The big disk on the left is the unit disk. The point of tangency of the two big disks is the point $\lambda=1$.
}
\end{figure}

The theory of perturbed Fatou coordinates (\cite{DouadyHubbard1985}, \cite{Shishikura1998})
plays an essential role in the proof of this theorem.
The proof relies on the fact that, truncating the series amounts essentially to
 compute the area of points that spend a long time in a bounded domain close to the filled Julia set.
For parameters close to the cusp of $\MandelbrotSet$,
 this set can contain many points in the basin of infinity.

The article is organised as follows.
Section \ref{sec:preliminaries}
 contains preliminary materials, including some results based on the bifurcation of parabolic fixed points
 and near parabolic fixed point theory.
Section \ref{sec:convergence properties of the approximation of the area}
 contains the main argument of the proof of theorem \ref{thm:discontinuity of area}
 and theorem \ref{thm:the actual main theorem}.

\subsection*{Acknowledgements} 
I would like to thank Prof. Xavier Buff for his help.
I would also like to thank Prof. Lasse Rempe-Gillen.

\section{Preliminaries}\label{sec:preliminaries}

\subsection{Some notations}

The following introduces some notations and basic facts.
For $r>0$, $\Disk_r$ denotes the disk of centre $0$ and radius $r$.

Let
$\qpl_{\lambda}(z)=\qplf{\lambda}{z}$ and
$\dblmand=\{\lambda\in\cplxp: \julia_\lambda\mbox{ is connected}\}$
be the double Mandelbrot set.

The B\"ottcher isomorphism of $\qpl_{\lambda}$ is denoted by
$\bott_\lambda:\BottcherDomain_\lambda\rightarrow \cplxp\wo \BottcherRangeComplementaryDisk_\lambda$.
It is defined in a neighbourhood $\BottcherDomain_\lambda$ of $\infty$
 in the Riemann sphere $\riesph$ onto the complementary of a closed
 disk $\BottcherRangeComplementaryDisk_\lambda$ with centre at $0$.
Its inverse will be denoted $\bottinv_\lambda$ and
$\green_\lambda$ will denote the Green function.

The notation
$\greencrit{g}=\{\lambda\in\cplxp,\green_\lambda(\critp_\lambda)\leq g\}$
will be used.
Thus $\dblmand=\greencrit{0}$.


For $\lambda\in\dblmand$,
$\GreenFunctionSublevelSet_{\lambda}({g})$ will be the set $\{z:\GreenFunction_\lambda(z)\leq g\}$.
The mapping $g\in\realsnn\mapsto\greensl_\lambda(g)$ is continuous with respect to the Hausdorff metric and
$\greensl_\lambda(0)=\fillj_\lambda$.
%

It is well known that the mapping
$\lambda\mapsto\fillj_\lambda$
is upper semi-continuous.
As a consequence,
$\ds{\limsup_{\lambda\tendsto 1}}\area{K_\lambda}\leq \area{K_1}$
(see also \ref{cor:limsup area of filled jset}).

Given $x\in\SetOfRealNumbers$, let's denote by $\FractionalPart{x}$ the fractional part of $x$,
 that is $\FractionalPart{x}=x-\IntegerPart{x}$
where $\IntegerPart{x}=\max\{n\in\SetOfIntegers:n\leq x\}$ is the integer part of $x$.
We will also need a notation for $\CeilingFunction{x}=\min\{n\in\SetOfIntegers:n\geq x\}$.

\subsection{Estimates on the Green function}

This section contains some elementary yet usefull estimates on the Green function.

\begin{lem}\label{lem:green estimates}
For all $\lambda$ such that $|\lambda|\leq 5$ and all $\xi\in\cplxp$,
\begin{equation*}
\log|\xi| - \log 6 \leq \green_\lambda(\xi) \leq \max\{\log 11, \log|\xi|+\log 11/6\}.
\end{equation*}

\end{lem}
\begin{proof}
 For $|\xi|\geq 6$, the value of $|\qpl_{\lambda}^{\iter{\lit}}(\xi)|$
 increases and tends to $\infty$ as
 $\lit\tendsto\infty$. For such $\xi$, we have
 $|\xi|^2\left(1+\left|\frac{\lambda}{\xi}\right|\right)
   \leq \qpl_\lambda(\xi) \leq
 |\xi|^2\left(1-\left|\frac{\lambda}{\xi}\right|\right)$.
Hence,
\begin{equation*}
\log|\xi|-\log 6 \left(1-\frac{1}{2^\lit}\right)
 \leq \frac{\log|\qpl_{\lambda}^{\iter{\lit}}(\xi)|}{2^\lit} \leq
 \log|\xi|-\log 11/6 \left(1-\frac{1}{2^\lit}\right).
\end{equation*}
This solve the case $|\xi|\leq 6$.
In the case where $|\xi|\leq 6$, the result follows from maximum principle for harmonic functions.
\end{proof}

\begin{lem}\label{lem:green and iterate}
Let $\escrad>6$, $g>0$ and $\lambda\in\cplxp$ such that $|\lambda|\leq 5$.

If $\lit=\ceilf{\frac{\log\log{11\escrad/6}-\log g}{\log 2}}\geq 0$
 and if $z\in\cplxp$ is such that $|\qpl_{\lambda}^{\iter{\lit}}(z)|\leq \escrad$
 then $|\green_\lambda(z)|\leq g$.
\end{lem}

\begin{proof}
Using lemma \ref{lem:green estimates} with $\xi=\qpl_{\lambda}^{\iter{\lit}}(z)$
 and the fact that $\green_\lambda\circ\qpl_{\lambda} = 2\green_\lambda$,
 we get
$\green_\lambda(z)\leq \frac{1}{2^\lit}\max\left\{\log 11,{\log 11 R
/6}\right\}$.
 The results then follows from the fact that $\escrad>6$ and the definition of $\lit$.
\end{proof}

\subsection{The area formula}

In the present situation, the area formula \eqref{eq:original area formula} can be restated as
\begin{equation}\label{gronwall area formula} 
\area{\greensl_\lambda(\log r)}=\pi\left(r^2-\sum_{k=1}^\infty k\frac{|\bic^{\lambda}_{k}|^2}{r^{2k}}\right),
\end{equation}
where
 $\bottinv_\lambda(w)=w+\ds{\sum_{k=1}^\infty}\frac{\bic^{\lambda}_{k}}{w^k}$, 
 $r\geq 1$
and
 $\lambda\in \greencrit{\log r}=\{\lambda\in\cplxp,\green_\lambda(\critp_\lambda)\leq \log r\}$.

Define, for $\lambda\in\cplxp$, $r>1$ and $\lis\in\SetOfPositiveNaturalNumbers$,
\begin{equation}\label{eq:def of truncated area formula}
\pa(\lambda,r,\lis) = \pi\left(r^2-\sum_{k=1}^\lis\frac{|\bic^{\lambda}_{k}|^2}{r^{2k}}\right)
\end{equation}
and
\begin{equation*}
\pa(\lambda,r,\infty) = \area{\greensl_\lambda(\log r)}.
\end{equation*}

Moreover, the following mappings are continuous:
\begin{itemize}
\item For $\lambda\in \dblmand$, $g\in\realsnn\mapsto \pa(\lambda,e^g,\infty)$,
\item For $g>0$, $\lambda\in\interior{\greencrit{g}}\mapsto\pa(\lambda,e^g,\infty)$,
\item For $\lis\in\SetOfPositiveNaturalNumbers$, $(\lambda,g)\in\cplxp\times\reals \mapsto \pa(\lambda,e^g,\lis)$.
\end{itemize}

The following lemmas will allow us to relate truncated approximations \eqref{eq:def of truncated area formula}
 with sublevel sets of the Green function.

\begin{lem}\label{lem:estim on area sums}
Let $\lis\in\SetOfPositiveNaturalNumbers$, $r>1$ and $\lambda\in\cplxp$ such that $\lambda\in\greencrit{\log r}$, then
\begin{equation*}
 \pa(\lambda,1,\lis)\geq \pi(1-r^{2\lis+2})+r^{2\lis}\pa(\lambda,r,\infty).
\end{equation*}

\end{lem}
\begin{proof}
Recall that $\pa(\lambda,1,\lis)=\pi\left(1-\ds{\sum_{k=1}^\lis}k|\bic^{\lambda}_{k}|^2\right)$.
From 
\begin{equation*}
 \sum_{k=1}^\lis k|\bic^{\lambda}_{k}|^2\leq r^{2\lis}\sum_{k=1}^{\lis}k\frac{|\bic^{\lambda}_{k}|^2}{r^{2k}}
 \leq r^{2\lis}\sum_{k=1}^{\infty}k\frac{|\bic^{\lambda}_{k}|^2}{r^{2k}},
\end{equation*}
it follows that 
$\pa(\lambda,1,\lis)\geq \pi\left(1-r^{2\lis}\ds{\sum_{k=1}^{\infty}}k\frac{|\bic^{\lambda}{k}|^2}{r^{2k}}\right)$,
which is equivalent to the statement of the lemma.
\end{proof}

\begin{lem}\label{lem:key lemma estim approx from below}
 Let $R>6$, $N\in\SetOfPositiveNaturalNumbers$, $r\in\OpenInterval{1,{11R/6}}$ and $\lambda\in\SetOfComplexNumbers$.
 Suppose that $|\lambda|<5$ and $\lambda\in \greencrit{\log r}$.
Define $\lit=\ceilf{\frac{\log\log{11\escrad/6}-\log\log r}{\log 2}}$.

Then $\lit\geq 0$ and
\begin{equation*}
 \pa(\lambda,1,\lis)\geq \pi(1-r^{2\lis+2})+r^{2\lis}\Area\left(\left\{z:|\qpl_\lambda^{\Iterated{\lit}}(z)|\leq R\right\}\right).
\end{equation*}

\end{lem}

\begin{proof}
 This lemma is a direct consequence of lemmas \ref{lem:green and iterate} and \ref{lem:estim on area sums}.
\end{proof}

\begin{lem}\label{lem:upper bound for area approximation}
Let $\lis\in\SetOfPositiveNaturalNumbers$, $r>1$ and $\lambda\in\greencrit{\log r}$. Then,
\begin{equation}
 \pa(\lambda,1,\AreaFormulaPartialSumUpperBound)
  \leq
 \Area\left({\GreenFunctionSublevelSet_{\log r}}\right)+\pi r^{2\AreaFormulaPartialSumUpperBound+2}.
\end{equation}
\end{lem}

\begin{proof}
Indeed, for any $r>1$,
\begin{equation*}
\pi\left(1-\sum_{k=1}^{\lis}k|\bic^{\lambda}_{k}|^2\right)
 \leq
\pi\left(r-\sum_{k=1}^{\lis}\frac{k|\bic^{\lambda}_{k}|^2}{r^{2k}}\right). 
\end{equation*}
But, since $\ds{\sum_{k=1}^\infty} k|b_k|^2\leq 1$, 
\begin{equation*}
 \sum_{k\geq \lis +1}\frac{k|\bic^{\lambda}_{k}|^2}{r^{2k}}\leq \frac{1}{r^{2(\lis+1)}}.
\end{equation*}
\end{proof}

\subsection{Bifurcation of parabolic fixed points}

We recall some results that appear in \cite{DouadyHubbard1985}, \cite{Douady1994} and \cite{Shishikura1998}.
A good introduction to the classical part of the theory can also be found in  \cite{Zinsmeister1997}.
See also \cite{BuffCheritat2012} for the Inou-Shishikura part.

\subsubsection{Convergence with domain}\label{sec:cv with domain}

We consider families of analytic maps $\ph_\lambda:\domainof{\ph_\lambda}\rightarrow\cplxp$
 defined on some open subsets  $\domainof{\ph_\lambda}$ of a Riemann surface $S$,
 where $\lambda$ belongs to a subset $\Lambda$ of a 
 metric topological space.
 Let $\lambda_0$ be in the closure of $\Lambda$.

Let $\psi:\domainof{\psi}\rightarrow\SetOfComplexNumbers$ be a holomorphic mapping
 and let $U$ an open subset of the domain $\domainof{\psi}$ of the mapping $\psi$.

In what follow we will say that
 a family  of mappings $\ph_\lambda:\domainof{\ph_\lambda}\rightarrow\cplxp$
 converge to $\psi$ on $U$ when $\lambda\tendsto \lambda_0$
 if, for all compact subset $K$ of $U$,
 there is a neighbourhood $N$ of $\lambda_0$ in $\Lambda\cup\{\lambda_0\}$
 such that, for all $\lambda\in N\wo\{\lambda_0\}$, $K\subset \domainof{\ph_\lambda}$,
 and $\ph_\lambda$ converges uniformly on $K$ to $\psi$ when $\lambda\tendsto\lambda_0$ inside $N\wo\{\lambda_0\}$,
see also \cite{Shishikura1998}.

Note that, if $(\ph_\lambda)_\lambda$ converges to a non constant $\psi$ on $U$ as $\lambda\tendsto\lambda_0$,
 then  for any non empty open subset $V\subset U$,
 and for $\lambda$ close enough to $\lambda_0$ (depending on $V$),
 the intersection $\psi(V)\cap \ph_\lambda(V)$
 has non empty interior.

We will use the following elementary lemma in conjunction with the above local uniform convergence.
\begin{lem}\label{lem:cv lemma}
 Let $X$ be a compact metric space and $Y$ a metric space.
 Let $f:X\rightarrow Y$ be a continuous mapping. Suppose that there is a compact $K$ of $Y$
 with non empty interior, such that $f(X)\subset K$
 and denote by $\Distance$ the distance on $Y$.

 Then, there exists an $\eps_0>0$ such that, for all
 $\eps\in\OpenInterval{0,\eps_0}$, there is a compact set $X_\eps$ with non empty interior such that
 for any mapping $g:X\rightarrow Y$ such that $\sup \{\Distance (g(x),f(x)), x\in X\}\leq \eps$,
 we have $g(X_\eps)\subset K$.
\end{lem}
\begin{proof}
For $\eps>0$, denote by $K_\eps$ the compact set $\{y\in K:\Distance(y,\Boundary K)\geq \eps\}$. 
 Let $\eps_0$ be the supremum of $\eps>0$ for which the set $K_\eps$ has non empty interior.

Let $\eps\in\OpenInterval{0,\eps_0}$. Pick $\eps'\in\OpenInterval{\eps,\eps_0}$ and let $X_\eps = f^{-1}(K_{\eps'})$.
 Since $f$ is continuous, $X_\eps$ has non empty interior.

Moreover, if $g:X_\eps\rightarrow Y$ is such that 
 $\sup \{\Distance (g(x),f(x)), x\in X_\eps\}\leq \eps$,
 then, for all $x\in X_\eps$, $\Distance(g(x),K_\eps')\leq \eps$.
 But if $x\in X_\eps$ were such that $g(x)\notin K$, we would have $\Distance(g(x),K_\eps')\geq\eps'>\eps$.
\end{proof}

\subsubsection{Fatou coordinates for parabolic and near parabolic mappings}

In this section we recall some important well known results on Fatou coordinates.

The following theorem is classical.
The bulk of these results can be found in \cite{Shishikura2000}.

In the present section, unless otherwise specified, phrases such as ``open sets'', ``Jordan domains'', etc. will refer to
 subsets of the complex plane.

\begin{thm}[Extended fatou coordinates]\label{prop:extended fatou coordinates}

 Let $\qpl_{1}(z)=z+z^2$,
 and let $\fillj_1$ be its filled Julia set.

Then, there exist
 an open set
 $\BPFPDomainOfExtendedRepellingFatouCoordinates_1$,
 Jordan domains
 $\BPFPAttractingPetal_1$,
 $\BPFPRepellingPetal_1$,
 points
 $\BPFPAttractingPetalReferencePoint$,
 $\BPFPRepellingPetalReferencePoint$,
 a simply connected domain
 $\BPFPUnionOfAllPetals_1$,
 holomorphic mappings
 $\AttractingFatouCoordinates_1:\Interior{\fillj_1}\rightarrow \SetOfComplexNumbers$,
 $\RepellingFatouCoordinates_1:\BPFPDomainOfExtendedRepellingFatouCoordinates_1\rightarrow \SetOfComplexNumbers$,
 $\InverseRepellingFatouCoordinates_1:\SetOfComplexNumbers\rightarrow \SetOfComplexNumbers$,
 simple arcs
 $\BPFPAttractingArc_1$,
 $\BPFPRepellingArc_1$,
 and
 closed Jordan domains
 $\BPFPAttractingCrescent_1$ and
 $\BPFPRepellingCrescent_1$,
 satisfying the following properties.

\begin{enumerate}
\item \textit{(Petals)}
 \begin{enumerate}
\item $\BPFPAttractingPetalReferencePoint\in \BPFPAttractingPetal_1$,\label{item:attracting petal ref point}
\item $\BPFPRepellingPetalReferencePoint\in\BPFPRepellingPetal_1$ and $\qpl_1(\BPFPRepellingPetalReferencePoint)\in\BPFPRepellingPetal_1$,
\label{item:repelling petal ref point}
 \item $\BPFPAttractingPetal_1\subset \Interior{\fillj_1}$, \label{item:the attracting petal is in the filled jset}
 \item $\BPFPUnionOfAllPetals_1=\BPFPAttractingPetal_1\cup\BPFPRepellingPetal_1\cup\{0\}$ is a neighbourhood of $0$,
\label{item:union of petals is neigh of 0}
 \item $\BPFPDomainOfExtendedRepellingFatouCoordinates_1=\ds{\bigcup_{n\geq 0}}\qpl_1^{-n}(\BPFPRepellingPetal_1)$.
 \end{enumerate}

\item \textit{(Properties of the mapping on the petals)} 
\begin{enumerate}
\item The mapping $\qpl_1$  is univalent on $\BPFPUnionOfAllPetals_1$,\label{item:qpl univ on petal}
\item $\qpl_1\left(\BPFPAttractingPetal_1\right)\subset \BPFPAttractingPetal_1$,\label{item:attracting petal invariant}
\item $\BPFPRepellingPetal_1\subset \qpl_1\left(\BPFPRepellingPetal_1\right)$, \label{item:repelling petal pre-invariant}
\item For all $z\in\Interior{\FilledJuliaSet_1}$, there is an $n\in\SetOfNaturalNumbers$ such that
 $\qpl_1^{\Iterated{n}}(z)\in\BPFPAttractingPetal_1$. \label{item:all orbits pass through attracting petal}
\end{enumerate}

\item \textit{(Fatou coordinates)}
\begin{enumerate}
\item The mapping $\AttractingFatouCoordinates_1$ is univalent on $\BPFPAttractingPetal_1$,\label{item:attracting fatou univalent}
\item The mapping $\RepellingFatouCoordinates_1$ is univalent on $\BPFPRepellingPetal_1$,\label{item:repelling fatou univalent}
\item The mapping
 $\InverseRepellingFatouCoordinates_1:
\SetOfComplexNumbers\rightarrow \InverseRepellingFatouCoordinates_1\left(\SetOfComplexNumbers\right)$
 coincides with
 the inverse of $\RepellingFatouCoordinates_1$
 on $\RepellingFatouCoordinates_1(\BPFPRepellingPetal_1)\cap\SetOfComplexNumbers$,
\item $\AttractingFatouCoordinates_1(\FilledJuliaSet_1)=\SetOfComplexNumbers$,\label{item:attracting fatou onto}
\item $\InverseRepellingFatouCoordinates_1(\SetOfComplexNumbers)=\SetOfComplexNumbers$.
\end{enumerate}

\item \textit{(Semi-conjugacy)} \label{item:fatou semi-conjugacy}
\begin{enumerate}
\item $\AttractingFatouCoordinates_1\ComposedWith\qpl_1=\AttractingFatouCoordinates_1+1$,
\item $\RepellingFatouCoordinates_1\ComposedWith\qpl_1=\RepellingFatouCoordinates_1+1$
 on $\qpl_1^{-1}\left(\BPFPRepellingPetal_1\right)\cap\BPFPRepellingPetal_1$.
\end{enumerate}

\item \textit{(Normalisation)}
\begin{enumerate}
\item $\AttractingFatouCoordinates_1(z)-\RepellingFatouCoordinates_1(z)\TendsTo 0$ as $z\TendsTo 0$
 with $z\in \BPFPAttractingPetal_1\cap \BPFPRepellingPetal_1$
 and $\ImaginaryPart{-1/z}\TendsTo+\infty$.\label{item:fatou repelling normalisation}
\end{enumerate}

\item \textit{(Crescents)} \label{item:crescents definition}
\begin{enumerate}
\item $\BPFPAttractingCrescent_1=
  \left(\AttractingFatouCoordinates_1\right)^{-1}\left(\{w:
      \RealPart{w}\in\ClosedInterval{\RealPart{\AttractingFatouCoordinates_1(\BPFPAttractingPetalReferencePoint)},
                                     \RealPart{\AttractingFatouCoordinates_1(\BPFPAttractingPetalReferencePoint)}+1}
  \}\right)$,\label{item:attracting crescent def}
\item $\BPFPRepellingCrescent_1=\InverseRepellingFatouCoordinates_1\left(\{w:\RealPart{w}\in
          \ClosedInterval{\RealPart\RepellingFatouCoordinates_1(\BPFPRepellingPetalReferencePoint),
                          \RealPart\RepellingFatouCoordinates_1(\BPFPRepellingPetalReferencePoint)+1}
      \}\right)$,\label{item:repelling crescent def}
\item $\BPFPAttractingCrescent_1 \subset \BPFPAttractingPetal_1\cup\{0\}$,
\item $\BPFPRepellingCrescent_1 \subset \BPFPRepellingPetal_1\cup\{0\}$, 
\item The arcs $\BPFPAttractingArc_1$ and $\BPFPRepellingArc_1$ join $0$ to itself,
\item The closed domain $\BPFPAttractingCrescent_1$ is bounded by $\BPFPAttractingArc_1$ and its image $\qpl_1\left(\BPFPAttractingArc_1\right)$,
\item The closed domain $\BPFPRepellingCrescent_1$ is bounded by $\BPFPRepellingArc_1$ and its image $\qpl_1\left(\BPFPRepellingArc_1\right)$,
\item $\BPFPAttractingCrescent_1\cap\BPFPRepellingCrescent_1=\{0\}$.

\end{enumerate}

\end{enumerate}

\end{thm}

\begin{proof}
From section 2.1 and propositions 2.21 and 3.2.3 of \cite{Shishikura2000}, there exists $\DomainOfInverseRepellingFatouMaxRealPart>0$,
 such that we have domains $\BPFPAttractingPetal_1$ and $\BPFPRepellingPetal_1$ and mappings
 $\AttractingFatouCoordinates_1:\BPFPAttractingPetal_1\rightarrow \SetOfComplexNumbers$
 and $\RepellingFatouCoordinates_1:\BPFPRepellingPetal_1\rightarrow\SetOfComplexNumbers$
 satisfying \ref{item:the attracting petal is in the filled jset}, \ref{item:union of petals is neigh of 0},
 \ref{item:qpl univ on petal}, \ref{item:attracting petal invariant}, \ref{item:repelling petal pre-invariant},
 \ref{item:attracting fatou univalent}, \ref{item:repelling fatou univalent}, \ref{item:fatou semi-conjugacy}
 and \ref{item:fatou repelling normalisation}.

Moreover we may suppose that 
\begin{align} 
 \BPFPAttractingPetal_1 & = \AttractingFatouCoordinates_1\left(\left\{
   w:|\arg (w-\DomainOfInverseRepellingFatouMaxRealPart)|<\frac{2\pi}{3}
 \right\}\right), \label{eq:definition of attracting petal}\\
 \BPFPRepellingPetal_1 & = \RepellingFatouCoordinates_1\left(\left\{
   w:|\arg(w+\DomainOfInverseRepellingFatouMaxRealPart)|>\frac{\pi}{3}
 \right\}\right),
\end{align}
where $\arg$ denotes the argument in $\OpenClosedInterval{-\pi,\pi}$.

Since vertical lines in $\left(\AttractingFatouCoordinates_1\right)^{-1}(\BPFPAttractingPetal_1)$
 are perpendicular to orbits, and since the basin of attraction of the parabolic fixed point $1$
 consists of points converging to $0$ from the direction $-1$ (compare \cite{B-Milnor2006}),
 any orbit converging non trivially to $0$ must intersect $\BPFPAttractingPetal_1$.
This shows \ref{item:all orbits pass through attracting petal}.

We can use a similar procedure as in section 4.2.3 of \cite{Shishikura2000} to extend the domain of the mappings
 $\AttractingFatouCoordinates_1:\BPFPAttractingPetal_1\rightarrow \SetOfComplexNumbers$
 and $\RepellingFatouCoordinates_1:\BPFPRepellingPetal_1\rightarrow\SetOfComplexNumbers$,
 respectively on $\Interior{\FilledJuliaSet_1}$ and $\BPFPDomainOfExtendedRepellingFatouCoordinates_1$.
These extended mappings still satisfy \ref{item:fatou semi-conjugacy}.

Let $w\in\SetOfComplexNumbers$. From equation \eqref{eq:definition of attracting petal},
 it follows that there exists an $n\geq 0$ such that $w+n\in\AttractingFatouCoordinates_1\left(\BPFPAttractingPetal_1\right)$.
 Let $z'\in\BPFPAttractingPetal_1$ such that $\AttractingFatouCoordinates_1(z')=w+n$.
Pick  $z\in\SetOfComplexNumbers$ such that $\qpl_1^{\Iterated{n}}(z)=z'$. Then, $z\in\FilledJuliaSet_1$ and
 $w=\AttractingFatouCoordinates_1(z')-n=\AttractingFatouCoordinates_1(z)$. Thus \ref{item:attracting fatou onto}.

Let's define $\InverseRepellingFatouCoordinates_1=\left(\RepellingFatouCoordinates_1\right)_{|\BPFPAttractingPetal_1}^{-1}$.
 The domain of the mapping $\InverseRepellingFatouCoordinates_1$ can be extended to $\SetOfComplexNumbers$ 
 (compare section 4.2.2 of Shishikura's article \cite{Shishikura2000}).
 The mapping $\InverseRepellingFatouCoordinates_1:\SetOfComplexNumbers\rightarrow\SetOfComplexNumbers$
 is onto since the mapping $\qpl_1$ is well defined from $\SetOfComplexNumbers$ to $\SetOfComplexNumbers$.

We can fix $\BPFPAttractingPetalReferencePoint\in\BPFPAttractingPetal_1$ and $\BPFPRepellingPetalReferencePoint\in\BPFPRepellingPetal_1$
 so that points \ref{item:attracting petal ref point} and \ref{item:repelling petal ref point}  are verified.
 Let's define $\BPFPAttractingCrescent_1$ and $\BPFPRepellingCrescent_1$ as in \ref{item:attracting crescent def} and \ref{item:repelling crescent def}.
 As in \cite{Shishikura2000}, it is possible to choose $\BPFPAttractingPetalReferencePoint$ and $\BPFPRepellingPetalReferencePoint$
 so that all the properties of item \ref{item:crescents definition} are satisfied.
\end{proof}

The following can be found in
 \cite{Shishikura1998}
 and has its origin in works of \'Ecalle, Lavaurs, Sentenac or Douady.
The modern version of this is at the basis of Inou-Shishikura near parabolic renormalisation theory
\cite{InouShishikura2008}.

\begin{thm}[Perturbed Fatou coordinates]
 \label{prop:perturbed fatou coordinates}

Let
$\qpl_{\lambda}(z)=\qplf{\lambda}{z}$
and,
 for $\lambda\neq 1$, denote by
$\{0,\TheOtherFixedPoint\}$
 the set of fixed points of $\qpl_{\lambda}$.
We will only consider $\lambda$ such that $\lambda\neq 1$
 and such that there is $\alpha\in \SetOfComplexNumbers$
 satisfying $|\arg \alpha|\leq \pi/4$ and $\lambda=e^{2\SquareRootOfMinusOne\pi\alpha}$.

Then, there exist a positive real number $\DomainOfInverseRepellingFatouMaxRealPart$
 and points $\BPFPAttractingPetalReferencePoint$ and $\BPFPRepellingPetalReferencePoint$
 such that the conclusions of the theorem \ref{prop:extended fatou coordinates}
 hold and such that we have the following.

There exists a positive real number $\alpha_0$, such that if $\lambda$ and $\alpha$ are as above with $|\alpha|\leq \alpha_0$,
 there exist
 open sets
 $\BPFPDomainOfExtendedAttractingFatouCoordinates_\lambda$,
 $\BPFPDomainOfExtendedRepellingFatouCoordinates_\lambda$,
 simply connected domains
 $\BPFPUnionOfAllPetals_\lambda$ and
 $\DomainOfInverseRepellingFatou_\lambda$,
 Jordan domains
 $\BPFPAttractingPetal_\lambda$,
 $\BPFPRepellingPetal_\lambda$,
 holomorphic mappings
 $\AttractingFatouCoordinates_\lambda:\BPFPDomainOfExtendedAttractingFatouCoordinates_\lambda\rightarrow\SetOfComplexNumbers$,
 $\RepellingFatouCoordinates_\lambda:\BPFPDomainOfExtendedRepellingFatouCoordinates_\lambda\rightarrow\SetOfComplexNumbers$,
 $\InverseRepellingFatouCoordinates_\lambda:\DomainOfInverseRepellingFatou_\lambda\rightarrow\SetOfComplexNumbers$,
 simple arcs
 $\BPFPAttractingArc_\lambda$,
 $\BPFPRepellingArc_\lambda$
 and closed Jordan domains
 $\BPFPAttractingCrescent_\lambda$ and
 $\BPFPRepellingCrescent_\lambda$
 satisfying the following.

\begin{enumerate}

\item (Petals)\label{item:perturbed petal}
\begin{enumerate}
\item $\BPFPAttractingPetalReferencePoint\in \BPFPAttractingPetal_\lambda$,\label{item:perturbed attracting petal ref point}
\item $\BPFPRepellingPetalReferencePoint\in\BPFPRepellingPetal_\lambda$
  and $\qpl_\lambda(\BPFPRepellingPetalReferencePoint)\in\BPFPRepellingPetal_\lambda$,
\label{item:perturbed repelling petal ref point}
\item The set $\BPFPUnionOfAllPetals_\lambda= \BPFPAttractingPetal_\lambda\cup\BPFPRepellingPetal_\lambda\cup\{0,\TheOtherFixedPoint\}$
 is a simply connected neighbourhood of $\{0,\TheOtherFixedPoint\}$,
\item $\BPFPDomainOfExtendedAttractingFatouCoordinates_\lambda=\ds{\bigcup_{n\geq 0}} \qpl_\lambda^{-n}(\BPFPAttractingPetal_\lambda)$,
\item $\BPFPDomainOfExtendedRepellingFatouCoordinates_\lambda=\ds{\bigcup_{n\geq 0}} \qpl_\lambda^{-n}(\BPFPRepellingPetal_\lambda)$,
\item $\DomainOfInverseRepellingFatou_\lambda= \left\{w:\RealPart w>\DomainOfInverseRepellingFatouMaxRealPart-\RealPart\frac{1}{\alpha} \right\}$.
\end{enumerate}

\item \label{item:perturbed qpl on petals}
\begin{enumerate}
\item The mapping $\qpl_\lambda$ is univalent on $\BPFPAttractingPetal_\lambda$ and on $\BPFPRepellingPetal_\lambda$.
\end{enumerate}

\item (Perturbed Fatou coordinates) \label{item:perturbed fatou coords}
\begin{enumerate}
\item The mapping $\AttractingFatouCoordinates_\lambda$ is univalent on $\BPFPAttractingPetal_\lambda$,
\item The mapping $\RepellingFatouCoordinates_\lambda$ is univalent on $\BPFPRepellingPetal_\lambda$,
\item The restriction of the mapping
 $\InverseRepellingFatouCoordinates_\lambda$
  on 
 $\left\{w:\RealPart w\in
   \OpenInterval{-\frac{1}{2\alpha}-\DomainOfInverseRepellingFatouMaxRealPart,
                 -\DomainOfInverseRepellingFatouMaxRealPart}
  \right\}$
 coincides with the inverse of $\RepellingFatouCoordinates_\lambda$ on its image.
\end{enumerate}

\item (Semi-conjugacy)\label{item:perturbed semi conj}
\begin{enumerate}
\item If $z\in\BPFPAttractingPetal_\lambda$ is such that $\qpl_\lambda(z)\in\BPFPAttractingPetal_\lambda$, then
 $\AttractingFatouCoordinates_\lambda(\qpl_\lambda(z))=\AttractingFatouCoordinates_\lambda(z)+1$,
\item If $z\in\BPFPRepellingPetal_\lambda$ is such that $\qpl_\lambda(z)\in\BPFPRepellingPetal_\lambda$, then
 $\RepellingFatouCoordinates_\lambda(\qpl_\lambda(z))=\RepellingFatouCoordinates_\lambda(z)+1$.
\end{enumerate}

\item (Normalisation)\label{item:perturbed normalisation}
\begin{enumerate}
\item There exists lifts $\widetilde{\AttractingFatouCoordinates_\lambda}$ and $\widetilde{\RepellingFatouCoordinates_\lambda}$
 of $\AttractingFatouCoordinates_\lambda$ and $\RepellingFatouCoordinates_\lambda$ by
 $w\mapsto \frac{\TheOtherFixedPoint}{1-e^{-2\SquareRootOfMinusOne\pi\alpha w}}$
 such that,
\begin{equation*}
 \widetilde{\RepellingFatouCoordinates_\lambda}\ComposedWith T_\alpha = T_\alpha\ComposedWith\widetilde{\AttractingFatouCoordinates_\lambda},
\end{equation*}
 where $T_\alpha$ denotes the translation by $-1/\alpha$.
\end{enumerate}

\item (Crescents)\label{item:perturbed crescents definition}
\begin{enumerate}
\item $\BPFPAttractingCrescent_\lambda=\left(\AttractingFatouCoordinates_\lambda\right)^{-1}
        \left(\{w:\RealPart{w}\in
              \ClosedInterval{\RealPart\AttractingFatouCoordinates_\lambda(\BPFPAttractingPetalReferencePoint),
                              \RealPart\AttractingFatouCoordinates_\lambda(\BPFPAttractingPetalReferencePoint)+1}
        \}\right)$,
\item $\BPFPRepellingCrescent_\lambda=\InverseRepellingFatouCoordinates_\lambda
        \left(\{w:\RealPart{w}\in
              \ClosedInterval{\RealPart\RepellingFatouCoordinates_\lambda(\BPFPRepellingPetalReferencePoint),
                              \RealPart\RepellingFatouCoordinates_\lambda(\BPFPRepellingPetalReferencePoint)+1}
        \}\right)$,
\item $\BPFPAttractingCrescent_\lambda\subset\BPFPAttractingPetal_\lambda\cup\{0,\TheOtherFixedPoint\}$,
\item $\BPFPRepellingCrescent_\lambda\subset\BPFPRepellingPetal_\lambda\cup\{0,\TheOtherFixedPoint\}$,
\item Both arcs $\BPFPAttractingArc_\lambda$ and $\BPFPRepellingArc_\lambda$ join $0$ to $\TheOtherFixedPoint$,
\item The closed domain $\BPFPAttractingCrescent_\lambda$ is bounded by $\BPFPAttractingArc_\lambda$
 and its image $\qpl_\lambda\left(\BPFPAttractingArc_\lambda\right)$,
\item The closed domain $\BPFPRepellingCrescent_\lambda$ is bounded by $\BPFPRepellingArc_\lambda$
 and its image $\qpl_\lambda\left(\BPFPRepellingArc_\lambda\right)$,
\item $\BPFPAttractingCrescent_\lambda\cap\BPFPRepellingCrescent_\lambda=\{0,\TheOtherFixedPoint\}$.
\end{enumerate}

\item \label{item:iterations through gate}
For all $z\in\BPFPAttractingCrescent_\lambda$, there is a $\lit\geq 1$ such that $\qpl_\lambda^{\Iterated{\lit}}(z)\in\BPFPRepellingCrescent_\lambda$,
 and for the smallest such $\lit$ we have:
\begin{equation}\label{eq:gate crossing}
 \RepellingFatouCoordinates_\lambda\left(\qpl_\lambda^{\Iterated{\lit}}(\zeta)\right) = 
 \AttractingFatouCoordinates_\lambda(\zeta)-\frac{1}{\alpha}+\lit,
\end{equation}
for all $\zeta$ in
 $\left(\AttractingFatouCoordinates_\lambda\right)^{-1}\left(\left\{
  w:\RealPart w\in\OpenInterval{\DomainOfInverseRepellingFatouMaxRealPart-\lit,\xi+\RealPart\frac{1}{2\alpha}}
 \right\}\right)\supset \BPFPAttractingPetal_\lambda$.

\item When $\lambda\TendsTo 1$ with $\lambda$ and $\alpha$ satisfying the above hypothesis, then we have the following convergences.
\begin{enumerate}
\item With respect to Hausdorff metric:\label{item:cv of closed domains}
\begin{enumerate}
\item $\BPFPAttractingArc_\lambda\TendsTo\BPFPAttractingArc_1$,
\item $\BPFPRepellingArc_\lambda\TendsTo\BPFPRepellingArc_1$,
\item $\BPFPAttractingCrescent_\lambda\TendsTo\BPFPAttractingCrescent_1$,
\item $\BPFPRepellingCrescent_\lambda\TendsTo\BPFPRepellingCrescent_1$.
\end{enumerate}

\item With respect to Hausdorff pseudometric:
\begin{enumerate}
\item $\BPFPAttractingPetal_\lambda\TendsTo\BPFPAttractingPetal_1$,
\item $\BPFPRepellingPetal_\lambda\TendsTo\BPFPRepellingPetal_1$,
\item $\BPFPUnionOfAllPetals_\lambda\TendsTo\BPFPUnionOfAllPetals_1$.
\end{enumerate}

\item As mappings with domains (compare section \ref{sec:cv with domain}):
\begin{enumerate}\label{item:cv of perturbed fatou coordinates}
\item $\AttractingFatouCoordinates_\lambda\TendsTo\AttractingFatouCoordinates_1$ on $\BPFPAttractingPetal_1$,
\item $\RepellingFatouCoordinates_\lambda\TendsTo\RepellingFatouCoordinates_1$ on $\BPFPRepellingPetal_1$,
\item $\InverseRepellingFatouCoordinates_\lambda\TendsTo\InverseRepellingFatouCoordinates_1$ on $\SetOfComplexNumbers$.
\end{enumerate}

\end{enumerate}

\end{enumerate}

\end{thm}

The perturbed Fatou coordinates is also called Fatou-Douady coordinates.

\begin{proof}
Let $\DomainOfInverseRepellingFatouMaxRealPart$ be as in theorem \ref{prop:extended fatou coordinates}.
One may increase the value of $\DomainOfInverseRepellingFatouMaxRealPart$ if necessary without changing the final conclusions.
Let $\PreFatouCoverMap_\lambda(w)=\frac{\TheOtherFixedPoint}{1-e^{-2\SquareRootOfMinusOne\pi\alpha w}}$
 and let $\PreFatouCoverMap_\lambda^{\AttractingTag}$ and $\PreFatouCoverMap_\lambda^{\RepellingTag}$
 be the restrictions of the mapping $\PreFatouCoverMap_\lambda$ on the respective domains \linebreak
 $\left\{w:\RealPart{w}\in\OpenInterval{-\RealPart\frac{1}{2\alpha}+\DomainOfInverseRepellingFatouMaxRealPart,
                                        \RealPart\frac{1}{2\alpha}+\DomainOfInverseRepellingFatouMaxRealPart}\right\}$
 and
 $\left\{w:\RealPart{w}\in\OpenInterval{-\RealPart\frac{1}{2\alpha}-\DomainOfInverseRepellingFatouMaxRealPart,
                                        \RealPart\frac{1}{2\alpha}-\DomainOfInverseRepellingFatouMaxRealPart}\right\}$.
The mappings $\PreFatouCoverMap_\lambda^{\AttractingTag}$ and $\PreFatouCoverMap_\lambda^{\RepellingTag}$
 are analytic diffeomorphisms onto their respective images.

Let $\BPFPAttractingPetal_\lambda$ be the image of
 $\left\{w:\RealPart (w-\DomainOfInverseRepellingFatouMaxRealPart)>-|\ImaginaryPart w|\right\}$
 by $\PreFatouCoverMap_\lambda^{\AttractingTag}$
 and $\BPFPRepellingPetal_\lambda$ the image of 
 $\left\{w:\RealPart (w+\DomainOfInverseRepellingFatouMaxRealPart)<|\ImaginaryPart w|\right\}$
 by $\PreFatouCoverMap_\lambda^{\RepellingTag}$.

Then, as in \cite{Shishikura2000} (see proposition 3.2.2 and sections 3.4.1, 3.4.3 and 4.2.3),
 the perturbed Fatou coordinates $\AttractingFatouCoordinates_\lambda$ and $\RepellingFatouCoordinates_\lambda$
 are defined on $\BPFPAttractingPetal_\lambda$ and $\BPFPRepellingPetal_\lambda$ respectively.
One can also define $\InverseRepellingFatouCoordinates_\lambda$ the same way as before.

One can easily check that the domain of these mappings can be extended so that items
 \ref{item:perturbed petal}, \ref{item:perturbed qpl on petals}, \ref{item:perturbed fatou coords} and \ref{item:perturbed semi conj}
 are true.

The item \ref{item:perturbed normalisation} follows from 3.4.3 of \cite{Shishikura2000}.

And the construction of crescents satisfying item \ref{item:perturbed crescents definition}
 is similar to the proof of theorem \ref{prop:extended fatou coordinates}.
\end{proof}

\begin{rem}\label{rem:attracting petal in disk6}
 One can choose $\DomainOfInverseRepellingFatouMaxRealPart>0$ so that $\BPFPAttractingPetal_1\subset\Disk_6$
 and $\BPFPAttractingPetal_\lambda\subset\Disk_6$
 for all $\lambda$ close enough to $1$.
\end{rem}

\subsubsection{Lavaurs maps}

The following result comes mainly from the work of Lavaurs \cite{Lavaurs1989}
 and can be found in Douady's paper \cite{Douady1994}.

Due to our purpose, the statement is only given with $\alpha$ real.

\begin{lem}[\cite{Douady1994}, section 1.8]\label{lem:cv to lavaurs map}
Let $\alpha_n\tendsto 0$ be a sequence of positive real numbers.
Suppose that the fractional part $\FractionalPart{\frac{1}{\alpha_n}}$ converges to 
 $\ebphase$.
Let
 $\lambda_n=e^{2\ii\pi\alpha_n}$ and
 $\lit_n=\frac{1}{\alpha_n}-\FractionalPart{\frac{1}{\alpha_n}}$.

Then $\qpl_{\lambda_n}^{\iter{\lit_n}}$ converges on compact subsets of $\Interior{\FilledJuliaSet_1}$ to the mapping
\begin{equation*}
 \LavaursMap_{\ebphase} = \InverseRepellingFatouCoordinates_{1}
\left(\AttractingFatouCoordinates_1-\ebphase
\right):\Interior{\FilledJuliaSet_1}\rightarrow \cplxp.
\end{equation*}

Moreover,
 $\LavaursMap_\EggBeaterPhase$ commutes with $\qpl_1$ and
 $\LavaursMap_\EggBeaterPhase(\BPFPAttractingPetal_1)=\SetOfComplexNumbers$.

\end{lem}

In \cite{Douady1994}, a mapping $\LavaursMap_{\ebphase}$ of the form
 $\LavaursMap_{\ebphase}=\InverseRepellingFatouCoordinates_{1}\left(\AttractingFatouCoordinates_{1}-\ebphase\right)$
 is called a Lavaurs map.

 The above proposition can be proven by using theorem \ref{prop:perturbed fatou coordinates}.

\begin{proof}

Let $\lambda$ be as in theorem \ref{prop:perturbed fatou coordinates}
 and let $z_\lambda\in\BPFPAttractingCrescent_\lambda\SetComplement\{0,\TheOtherFixedPoint\}$.
Then, from item \ref{item:iterations through gate} of that theorem,
 there exists $\lit'=\lit'(\lambda,z_\lambda)\geq 1$ minimal such that 
 $\qpl_\lambda^{\Iterated{\lit'}}(z_\lambda)\in\BPFPRepellingCrescent_\lambda$.
Moreover, for this $\lit'$, the identity
 $\qpl_\lambda^{\Iterated{\lit'}}(z)=\InverseRepellingFatouCoordinates_\lambda\left(
  \AttractingFatouCoordinates_\lambda(z)-\frac{1}{\alpha}+\lit'\right)$
 is true for all $z$ in a domain that converges to $\Interior{\FilledJuliaSet_1}$
 as $\lambda\TendsTo 1$.

For $\lambda$ fixed and $z$ close to $0$, the mapping $\qpl_\lambda^{\Iterated{k}}(z)$
 is close to $z\mapsto \lambda^{k}z$.
 Since the closed sets $\BPFPAttractingCrescent_\lambda$ and $\BPFPRepellingCrescent_\lambda$
 converge respectively to $\BPFPAttractingCrescent_1$ and $\BPFPRepellingCrescent_1$
 as $\lambda\TendsTo 1$, it follows that if $z_\lambda\in \BPFPAttractingCrescent_\lambda$ is close enough to $0$,
 and if $\lit'\geq 1$ is minimal such that $\qpl_\lambda^{\Iterated{\lit'}}(z_\lambda)\in\BPFPRepellingCrescent_\lambda$,
 then, $\lit'<\frac{1}{\alpha}$.
 
Moreover, for this $\lit'(\lambda)$, equation \eqref{eq:gate crossing} is verified.
 Since $\lit'(\lambda)$ does not depend on $z$, it follows from the convergence part of theorem
 \ref{prop:perturbed fatou coordinates}
 that $\frac{1}{\alpha}-\lit'$ is bounded.
 
Now, consider the sequence of $\lambda_n$ and
 suppose that $z_{\lambda_n}$ is chosen so that $\lit'(\lambda_n,z_{\lambda_n})<\frac{1}{\alpha_n}$.
 Then $\lit_n=\lit_n'+k_n$ with $\lit_n'=\lit'(\lambda_n,z_{\lambda_n})$ and $k_n\geq 0$.
Since the sequence $(k_n)_n$ is bounded, let $k_{max}$ be an upper bound for $k_n$.

Let $z\in\Interior{\FilledJuliaSet_1}$. Then, for $n$ large enough,
 $z,\qpl_{\lambda_n}(z),\dots,\qpl_{\lambda_n}^{\Iterated{k_{max}}}(z)$
 all belong to
 $\left(\AttractingFatouCoordinates_{\lambda_n}\right)^{-1}\left(\left\{
  w:\RealPart w\in\OpenInterval{\DomainOfInverseRepellingFatouMaxRealPart-\lit_n',\xi+\frac{1}{2\alpha_n}}
 \right\}\right)$.
Then,
\begin{align*}
 \qpl_{\lambda_n}^{\Iterated{\lit_n}}(z) & = \InverseRepellingFatouCoordinates_{\lambda_n}\left(
        \AttractingFatouCoordinates_{\lambda_n}\left(\qpl_{\lambda_n}^{\Iterated{k_n}}(z)\right)
         -\frac{1}{\alpha_n}+\lit_n'
 \right) \\
 & = \InverseRepellingFatouCoordinates_{\lambda_n}\left(
        \AttractingFatouCoordinates_{\lambda_n}(z)
         -\frac{1}{\alpha_n}+\lit_n
 \right).
\end{align*}

Finally, the last part follows from the fact that 
 $\AttractingFatouCoordinates_1$ and $\InverseRepellingFatouCoordinates_1$
 are onto.
\end{proof}

\begin{lem}\label{lem:domain exists}
There exists $\AreaGap>0$ and
 $\TransitTimeMinusInverseRotationNumberBound>0$ such that the following holds.

\begin{enumerate}
\item
Let $\escrad>6$
 and let $(\alpha_n)_n$ be a sequence of real numbers such that
 $\alpha_n\tendsto 0$.

Then 
\begin{equation*}
\Area(\FilledJuliaSet_1)>\Area(\FilledJuliaSet_{\alpha_n}) + \AreaGap.
\end{equation*}

\item
More precisely, if the fractional part $\FractionalPart{\frac{1}{\alpha_n}}$ converges as $n\TendsTo\infty$,
then there exists a non empty open set $V$, relatively compact in $\Interior{\FilledJuliaSet_1}$,
 an integer $n_0$
 and a sequence of whole numbers $(\lit_n)_{n\geq n_0}$
 such that for all $n\geq n_0$,
\begin{enumerate}
\item the set $V$ is a subset of the basin of attraction of $\infty$ of $\qpl_{e^{2i\pi\alpha_n}}$, \label{item:v in A infty}
\item $\qpl_{e^{2i\pi\alpha_n}}^{\iter{\lit_n}}(V)\subset \disk_{\escrad}$,
\item $\left|\lit_n-\frac{1}{\alpha_n}\right|< \TransitTimeMinusInverseRotationNumberBound$,
\item $\Area(V)\geq\AreaGap$. \label{item:unif area lower bound}
\end{enumerate}

\end{enumerate}

\end{lem}
\begin{proof}
Choose a non empty open subset $\Omega$,
 relatively compact in the intersection of the basin of attraction of $\infty$ of $\qpl_1$
 with the disk of radius $\escrad$ centered at $0$.

Let $\EggBeaterPhase\in\ClosedInterval{0,1}$ and
 let's first suppose that the sequence of fractional parts $\FractionalPart{\frac{1}{\alpha_n}}$ converges to $\EggBeaterPhase$.

Let $V=\LavaursMap_{\EggBeaterPhase}^{-1}(\Omega)$.
 By definition $V$ is a non empty open subset
 relatively compact in $\Interior{\FilledJuliaSet_1}$.

The point \ref{item:v in A infty} follows from the fact that $\alpha_n\tendsto 0$.
The two following points then follow from lemma \ref{lem:cv to lavaurs map} by letting
 $\lit_n=\frac{1}{\alpha_n}-\FractionalPart{\frac{1}{\alpha_n}}$ and $P=2$.
 
Finally the uniform lower bound \ref{item:unif area lower bound} on the area comes form the continuity of the mapping $\EggBeaterPhase\mapsto\Area(V)$.
 And the first statement of the lemma follows from this estimate.
\end{proof}

Theorem \ref{thm:discontinuity of area} is simply a corollary of the previous lemma.

\begin{lem}\label{lem:seq of domain by lavaurs}
Let $\EggBeaterPhase\in\ClosedInterval{0,1}$.
 Let $\EscapingDomain$ be a compact set with non empty interior included in the domain of attraction of $\infty$
 of $\qpl_1$
 and such that
 $\EscapingDomain\subset\BPFPRepellingCrescent_1$.
 Then, there exists a sequence of compacts sets $\EscapingDomain^j$ with
 $\EscapingDomain^0=\EscapingDomain$ and such that the following is true for all $j\geq 1$,
\begin{enumerate}
 \item $\Interior{\EscapingDomain^j}\neq \EmptySet$,
 \item $\EscapingDomain^j\subset \FilledJuliaSet_1$,
 \item $\LavaursMap_\EggBeaterPhase(\EscapingDomain^j) \subset {\EscapingDomain^{j-1}}$,
\end{enumerate}
where $\LavaursMap_\EggBeaterPhase=\InverseRepellingFatouCoordinates_{1}\left(\AttractingFatouCoordinates_{1}-\ebphase\right)$
 is a Lavaurs map.
\end{lem}

\begin{proof}
Let, for $j\geq 0$, $\EscapingDomain^{j+1}=\LavaursMap_\EggBeaterPhase^{-1}(\EscapingDomain^{j})$.
 By construction, $\EscapingDomain^{j+1}\subset \Interior{\FilledJuliaSet_1}$.
 Moreover, since $\LavaursMap_\EggBeaterPhase$ is onto, $\EscapingDomain^{j+1}$
 has non empty interior whenever $\EscapingDomain^{j}$ has non empty interior.\end{proof}

\begin{lem}\label{lem:existence of seq of lows escaping domains}
Let $\EggBeaterPhase\in\ClosedInterval{0,1}$ and $\LogGrowthOrder\geq 0$.

Let $\EscapingDomain\subset {\BPFPRepellingCrescent_1}$ 
 be a compact set with non empty interior contained in the basin of attraction of $\infty$
 of $\qpl_1$.

Then, there exists $\AlphaNeighbourhoodBound>0$ and a finite sequence $(\EscapingDomain^j)_{j\leq\LogGrowthOrder}$
 of compact sets $\EscapingDomain^j\subset \Interior{\FilledJuliaSet_1}$ with non empty interior, such that $\EscapingDomain^0=\EscapingDomain$ and
 satisfying the following.

For all sequence of $\lambda_n=e^{2\SquareRootOfMinusOne\pi\alpha_n}$ with $\alpha_n\in\OpenClosedInterval{0,\AlphaNeighbourhoodBound}$,
 such that $\FractionalPart{1/\alpha_n}$ converges to $\EggBeaterPhase$,
 there exists a $n_0\geq 0$ such that,
 for all $n\geq n_0$
 and all $j \in\{1,\dots,\CeilingFunction{\LogGrowthOrder}\}$,
\begin{equation*}
\qpl_{\lambda_n}^{\Iterated{\BigIterate_n}}(\EscapingDomain^j) \subset \EscapingDomain^{j-1}. 
\end{equation*}

\end{lem}
\begin{proof}
This follows from lemmas \ref{lem:seq of domain by lavaurs}, \ref{lem:cv to lavaurs map} and \ref{lem:cv lemma}.
\end{proof}

\section{Convergence properties of the approximation of the area}\label{sec:convergence properties of the approximation of the area}

Theorem \ref{thm:discontinuity of area} is a direct consequence of propositions \ref{prop:first estimate}
 and \ref{prop:convergence to area of K1} below.


%
%
%


\begin{prop} \label{prop:first estimate}

There exists a constant $\AreaGap >0$ satisfying the following.

Let $(\alpha_n)_n$ be a sequence of positive real numbers
 converging to $0$ and $(\lis_n)_n$ be a sequence of whole numbers.
Let $\lambda_n=e^{2\ii\pi\alpha_n}$.

Suppose that $\lis_n=\SmallO{2^{1/\alpha_n}}$,
 then there exists $n_0\in\SetOfNaturalNumbers$ such that, for all $n\geq n_0$,
\begin{equation*}
 \pa(\lambda_n, 1, \lis_n) \geq \AreaGap + \area{K_{\lambda_n}}.
\end{equation*}

\end{prop}

\begin{proof}
One must notice first that $\lambda_n$ is such that $|\lambda_n|<5$ and belongs to
 $\greencrit{g}=\{\lambda\in\cplxp,\green_\lambda(\critp_\lambda)\leq g\}$
 for any $g\geq 0$.

For any subsequence $(\alpha_{n_j})_j$ such that $\FractionalPart{\frac{1}{\alpha_{n_j}}}$ converges,
 let $V$ and $\TransitTimeMinusInverseRotationNumberBound$
 be the open set and positive constant defined by lemma \ref{lem:domain exists}. 
 Let $\EscapeRadius>6$.
According lemma \ref{lem:domain exists}, $V$ does not intersect the filled Julia set $\FilledJuliaSet_{\lambda_n}$
 and there is a sequence $(\lit_n)_n$
 of natural numbers converging to $\infty$ such that, for all $n\geq n_0$ big enough,
 $\qpl_{\lambda_n}^{\iter{\lit_n}}(V)\subset \Disk_{\EscapeRadius}$.

For $n\geq n_0$, let $g_n=\left(\log\frac{11\EscapeRadius}{6}\right)2^{-\lit_n}$
 and $r_n=e^{g_n}$.
By lemma \ref{lem:key lemma estim approx from below},
\begin{equation*}
 \pa(\lambda_n,1,\lis_n)\geq\pi\left(1-r_n^{2\lis_n+2}\right)+r_n^{2\lis_n}\left(\Area{\FilledJuliaSet_{\lambda_n}}+\Area{V}\right).
\end{equation*}

By lemma \ref{lem:domain exists}, $\lit_n$ is at a bounded distance from $\frac{1}{\alpha_n}$,
 and since \linebreak $N_n=\SmallO{2^{1/\alpha_n}}$, $r_n^{2\lit_n}=\exp\left(2^{-\lit_n+1}\lis_n\right)$ tends to $1$.
 This implies the proposition.
\end{proof}

\begin{prop}\label{prop:convergence to area of K1}

If $\lis_n=\SmallO{2^{1/\alpha_n}}$,
 then
\begin{equation*}
 \pa(\lambda_n, 1, \lis_n) \TendsTo \Area(\FilledJuliaSet_1).
\end{equation*}
\end{prop}

\begin{proof}
For $\lis$ fixed, $\pa(\lambda_n, 1, \lis)$ converges to  $\pa(1, 1, \lis)$ when $n\TendsTo\infty$.
Moreover $\pa(1, 1, \lis)$ converges to $\Area(\FilledJuliaSet_1)$ when $N$ grows, thus
\linebreak
$\ds{\limsup_{n\TendsTo\infty}}\pa(\lambda_n, 1, \lis_n)\leq\Area(\FilledJuliaSet_1)$.

Now we show that $\ds{\liminf_{n\TendsTo\infty}}\pa(\lambda_n, 1, \lis_n)\geq\Area(\FilledJuliaSet_1)$.
Let, for $\eps>0$, $\FilledJuliaSet_1^{\eps}=\left\{z\in\FilledJuliaSet_1:d(z,\TopologicalBoundary{\FilledJuliaSet_1})\geq\eps\right\}$.
Since, for $\eps$ small enough,
 $\FilledJuliaSet_1^{\eps}$ is a compact subset of the interior of $\FilledJuliaSet_1$,
 there exists $\IterationsToThePetal\in\SetOfNaturalNumbers$ such that
 $\qpl_1^{\Iterated \IterationsToThePetal}(\FilledJuliaSet_1^{\eps})\subset \BPFPAttractingPetal_1$.
Moreover there exists $\MaxDifferenceInRealPartFatouCoordinate>0$ such that
\begin{equation*}
\max \left\{
\left|\RealPart(\AttractingFatouCoordinates_\lambda(z))-\RealPart(\AttractingFatouCoordinates_\lambda(\BPFPAttractingPetalReferencePoint))\right|,
z\in\qpl_1^{\Iterated \IterationsToThePetal}(\FilledJuliaSet_1^{\eps})
\right\}\leq \MaxDifferenceInRealPartFatouCoordinate.
\end{equation*}
We may increase the value of $\IterationsToThePetal$ and $\MaxDifferenceInRealPartFatouCoordinate$
 so that the above is true also for $\qpl_{\lambda_n}$ for all $n$ big enough.

From point \ref{item:iterations through gate} of theorem \ref{prop:perturbed fatou coordinates},
 it follows that there exists a constant $\TransitTimeMinusInverseRotationNumberBound$ independent of $n$
 such that for all $\lit\leq\frac{1}{\alpha_n}-\TransitTimeMinusInverseRotationNumberBound$,
 $\qpl_{\lambda_n}^{\Iterated p}(\FilledJuliaSet_1^\eps)\subset\Disk_{6}$.

As previously one concludes that, if
 $r_n=\left(\frac{11\EscapeRadius}{6}\right)^{2^{-\frac{1}{\alpha_n}-\TransitTimeMinusInverseRotationNumberBound}}$,
 then
\begin{equation*}
 \pa(\lambda_n, 1, \lis_n)\geq\pi\left(1-r_n^{2\lis_n+2}\right)+r_n^{2\lis_n}\Area\FilledJuliaSet_1^\eps.
\end{equation*}
Hence $\ds{\liminf}\pa(\lambda_n, 1, \lis_n)\geq \Area\FilledJuliaSet_1^\eps$ for any $\eps>0$.
\end{proof}

\begin{prop}\label{prop:second estimate}
Let $\LogGrowthOrder>0$ and $\EggBeaterPhase\in\ClosedInterval{0,1}$.
Then there exists $\AreaGap>0$ satisfying the following.

Let $(\alpha_n)_n$ be a sequence of positive real numbers
 converging to $0$ and such that $\IntegerPart{\frac{1}{\alpha_n}}\TendsTo \EggBeaterPhase$,
 and let $(\lis_n)_n$ be a sequence of whole numbers.
Define $\lambda_n=e^{2\ii\pi\alpha_n}$.

Suppose that $\log \lis_n\leq \frac{\LogGrowthOrder}{\alpha_n}$.

Then, there exists $n_0\in\SetOfNaturalNumbers$ such that, for all $n\geq n_0$,
\begin{equation*}
 \pa(\lambda_n, 1, \lis_n) \geq \AreaGap + \Area{\FilledJuliaSet_{\lambda_n}}.
\end{equation*}

\end{prop}
\begin{proof}
Define $\LogGrowthOrder'=\CeilingFunction{\LogGrowthOrder}+1$.
According to lemma \ref{lem:existence of seq of lows escaping domains}, there exists a  finite sequence
 $(\EscapingDomain^j)_{j\leq\LogGrowthOrder'}$
 of compact sets $\EscapingDomain^j\subset \Interior{\FilledJuliaSet_1}$ with non empty interior,
 such that $\EscapingDomain^0$ is a subset of $\Disk_6$ (see remark \ref{rem:attracting petal in disk6})
 and of the basin of attraction of $\infty$ of $\qpl_{\lambda_n}$, for all $n\geq n_0$ (for some $n_0$),
 and such that $\qpl_{\lambda_n}^{\Iterated{\BigIterate_n}}(\EscapingDomain^j) \subset \EscapingDomain^{j-1}$,
 for $j=1,\dots,\LogGrowthOrder'$.

Then, from lemma \ref{lem:key lemma estim approx from below}, it follows that
\begin{equation*}
\pa(\lambda_n, 1, \lis_n)\geq
     \pi\left(1-r_n^{2\lis_n+2}\right)
     +  r_n^{2\lis_n}
          \left( \Area{\FilledJuliaSet_{\lambda_n}} + \Area{\EscapingDomain^{\LogGrowthOrder'}} \right),
\end{equation*}
with $\log r_n = \log\frac{11\EscapeRadius}{6}\cdot 2^{-\lit_n\LogGrowthOrder'}$ and $\lit_n=\IntegerPart{\frac{1}{\alpha_n}}$.

Since
\begin{equation*}
 \log\lis_n- \lit_n\LogGrowthOrder'\leq \LogGrowthOrder\left(\frac{1}{\alpha_n}-\IntegerPart{\frac{1}{\alpha_n}}\right)-\IntegerPart{\frac{1}{\alpha_n}}
 \TendsTo-\infty,
\end{equation*}
the above inequality implies the proposition.
\end{proof}

\bibliographystyle{plain}
\bibliography{./all}

\def\polhk#1{\setbox0=\hbox{#1}{\ooalign{\hidewidth
  \lower1.5ex\hbox{`}\hidewidth\crcr\unhbox0}}} \def\cprime{$'$}
\begin{thebibliography}{10}

\bibitem{Astala1994}
Kari Astala.
\newblock Area distortion of quasiconformal mappings.
\newblock {\em Acta Mathematica}, 173(1):37--60, 1994.

\bibitem{Holbaek2012}
Anna~M. Benini, Asl{\i} Deniz, Thomas Gauthier, and Tan Lei.
\newblock Quo vadis? open problems related to {MLC}.
\newblock 2012.

\bibitem{BuffCheritat2012}
Xavier {Buff} and Arnaud {Ch\'eritat}.
\newblock {Quadratic Julia sets with positive area.}
\newblock {\em {Ann. Math. (2)}}, 176(2):673--746, 2012.

\bibitem{Douady1994}
Adrien Douady.
\newblock Does the julia set depend continuously on the polynomial?
\newblock {\em Proc.Symp.App.Math.}, (49), 1994.

\bibitem{DouadyHubbard1985}
Adrien Douady and John~Hamal Hubbard.
\newblock {\em \'{E}tude dynamique des polyn\^omes complexes.}, volume 84-85 of
  {\em Publications Math\'ematiques d'Orsay [Mathematical Publications of
  Orsay]}.
\newblock Universit\'e de Paris-Sud, D\'epartement de Math\'ematiques, Orsay,
  1984-85.

\bibitem{EwingSchober1992}
John~H. Ewing and Glenn Schober.
\newblock The area of the {M}andelbrot set.
\newblock {\em Numer. Math.}, 61(1):59--72, 1992.

\bibitem{Fatou1919}
Pierre Fatou.
\newblock Sur les \'equations fonctionnelles.
\newblock {\em Bull. Soc. Math. France}, 47:161--271, 1919.

\bibitem{Fatou1920a}
Pierre Fatou.
\newblock Sur les \'equations fonctionnelles.
\newblock {\em Bull. Soc. Math. France}, 48:33--94, 1920.

\bibitem{GraczykSwiatek1997}
Jacek Graczyk and Grzegorz {\'S}wiatek.
\newblock Generic hyperbolicity in the logistic family.
\newblock {\em Ann. of Math. (2)}, 146(1):1--52, 1997.

\bibitem{GraczykSwiatekBook1998}
Jacek Graczyk and Grzegorz {\'S}wi{\polhk{a}}tek.
\newblock {\em The real {F}atou conjecture}, volume 144 of {\em Annals of
  Mathematics Studies}.
\newblock Princeton University Press, Princeton, NJ, 1998.

\bibitem{Gronwall1914}
Thomas~Hakon Gronwall.
\newblock Some remarks on conformal representation.
\newblock {\em Annals of Mathematics}, 16(1/4):pp. 72--76, 1914.

\bibitem{InouShishikura2008}
Inou Hiroyuki and Shishikura Mitsuhiro.
\newblock The renormalization for parabolic fixed points and their
  perturbation.
\newblock {\em preprint}, 2008.

\bibitem{Lavaurs1989}
Pierre Lavaurs.
\newblock {\em Syst\`emes dynamiques holomorphes: explosion de points
  p{\'e}riodiques paraboliques}.
\newblock PhD thesis, Universit{\'e} Paris-Sud, 1989.

\bibitem{Lyubich1983}
Mikhail~Yu Lyubich.
\newblock On typical behavior of the trajectories of a rational mapping of the
  sphere.
\newblock {\em Soviet Math. Dokl.}, (27):22--25, 1983.

\bibitem{Lyubich1991}
Mikhail~Yu Lyubich.
\newblock On the lebesgue measure of the julia set of a quadratic polynomial.
\newblock {\em Preprint IMS at Stony Brook}, (10), 1991.

\bibitem{McMullenBook1994}
Curtis~T. McMullen.
\newblock {\em Complex dynamics and renormalization}, volume 135 of {\em Annals
  of Mathematics Studies}.
\newblock Princeton University Press, Princeton, NJ, 1994.

\bibitem{B-Milnor2006}
John Milnor.
\newblock {\em Dynamics in one complex variable}, volume 160 of {\em Annals of
  Mathematics Studies}.
\newblock Princeton University Press, Princeton, NJ, third edition, 2006.

\bibitem{PetersenZakeri2004}
Carsten~Lunde Petersen and Saeed Zakeri.
\newblock On the julia set of a typical quadratic polynomial with a siegel
  disk.
\newblock {\em Annals of Mathematics}, 159(1):pp. 1--52, 2004.

\bibitem{Shishikura1998}
Mitsuhiro Shishikura.
\newblock The hausdorff dimension of the boundary of the mandelbrot set and
  julia sets.
\newblock {\em Annals of Mathematics}, 147(2):pp. 225--267, 1998.

\bibitem{Shishikura2000}
Mitsuhiro Shishikura.
\newblock Bifurcation of parabolic fixed points.
\newblock In {\em The {M}andelbrot set, theme and variations}, volume 274 of
  {\em London Math. Soc. Lecture Note Ser.}, pages 325--363. Cambridge Univ.
  Press, Cambridge, 2000.

\bibitem{Zinsmeister1997}
Michel Zinsmeister.
\newblock Basic parabolic implosion in five days (after {A}. {D}ouady).
\newblock 1997.
\newblock Course notes.

\end{thebibliography}


\end{document}